\documentclass[pdftex]{imsart}
\usepackage{amsthm,amsmath,amssymb,epic,latexsym,multirow,ulem,dsfont,mathrsfs}

\usepackage[numbers]{natbib}
\RequirePackage{natbib}
\RequirePackage[colorlinks,citecolor=blue,urlcolor=blue]{hyperref}
\usepackage{comment}
\usepackage{wrapfig}
\usepackage[english]{babel}
 \usepackage[T1]{fontenc} 
 \usepackage[ddmmyyyy,hhmmss]{datetime}
\usepackage[svgnames,dvipsnames]{xcolor}
 \usepackage{makeidx}
\usepackage{epsf}
\usepackage[dvips]{epsfig}
\usepackage{subfig}
 \usepackage{graphicx}
\usepackage{nicefrac}

\theoremstyle{plain}
\newtheorem{theo}{Theorem}[section]
\newtheorem*{theo*}{Theorem}
\newtheorem{cor}[theo]{Corollary}
\newtheorem{prop}[theo]{Proposition}
\newtheorem{lem}[theo]{Lemma}

\newtheorem*{defi*}{Definition}
\newtheorem{rem}[theo]{Remark}

\newtheorem{nota}[theo]{Notations}

\newcommand{\R}{\mathbb{R}}

\newcommand{\Z}{\mathbb{Z}}
\newcommand{\E}{\mathbb{E}}
\renewcommand{\P}{\mathbb{P}}

\newcommand{\red}[1]{{\color{black}{#1}}} 

\newcommand{\mB}{\mathfrak{R}}

\newcommand{\un}{\mathds{1}}

\newcommand{\ox}{\overline{x}} 
\newcommand{\aXn}[1][]{|X_{#1}|}
\newcommand{\ow}{\overline{w}}
\newcommand{\oX}{\overline{X}} 
\newcommand{\oz}{\overline{z}} 
\newcommand{\oy}{\overline{y}}

\newcommand{\cst}{\frac{8}{\pi}}

\newcommand{\pun}{ p^{\dag}}

\newcommand{\co}{ {\mathbf{c_{_1}}}}

\newcommand{\cun}{ {\mathbf{c_{_1}}}}

\begin{document}

\title{Sub-diffusive behavior of a recurrent Axis-Driven Random Walk}



\author{\fnms{Pierre} \snm{Andreoletti}}
\address{
Institut Denis Poisson,  UMR C.N.R.S. 7013, Orl\'eans, France
}
\and 
\author{\fnms{Pierre} \snm{Debs}}
\address{Institut Denis Poisson,   UMR C.N.R.S. 7013,
Universit\'e d'Orl\'eans, Orl\'eans, France.} 
\runauthor{Andreoletti, Debs}

\begin{abstract}
We study the second order of the number of excursions of a simple random walk with a bias that drives a return toward the origin along  the axes introduced by P. Andreoletti and P. Debs \cite{AndDeb3}. This is a crucial step toward deriving the asymptotic behavior of these walks, whose limit is explicit and reveals various characteristics of the process: the invariant probability measure of the extracted coordinates away from the axes, the 1-stable distribution arising from the tail distribution of entry times on the axes, and finally, the presence of a Bessel process of dimension 3, which implies that the trajectory can be interpreted as a random path conditioned to stay within a single quadrant.
\end{abstract}

\maketitle

\section{Introduction}

Inhomogeneous random walks where the transition probabilities depend on the walker’s position have received considerable attention, with models ranging, for example, from oscillating walks \cite{Kemperman} to random walks in random environments \cite{Zeitouni}. A notable feature in these models is the disruption of the uniform transition behavior typical of simple random walks on graphs like $\mathbb{Z}^2$.

In this paper, we continue the investigation initiated in \cite{AndDeb3} and \cite{Pierre18}, where such random walk was introduced to model the movement of a particle subjected to a localized force field,  applied along specific spatial directions, while the particle remains free to diffuse elsewhere. More precisely in \cite{AndDeb3}, the model considers a nearest-neighbor random walk on $\mathbb{Z}^2$, where a restoring force is applied solely on the axes, and a simple random walk governs the motion off the axes. The intensity of this force depends on the particle's coordinates and a parameter $\alpha$, assumed large in both \cite{AndDeb3} and the present work, to emphasize the particle's concentration toward the origin.

This earlier work established the existence of two invariant probability measures characterizing the entrance and exit behavior of the walk along the axes. This led to a renewal theorem and an ergodic description of the particle's asymptotic behavior in that constrained region.

Here, we shift our focus to the behavior of the walk in the cones, the domains between the axes, where the walk is unaffected by the external force and thus diffuses freely. While the dynamics in the constrained region have been well described, understanding the long-time behavior in the unconstrained zones reveals a subtler, subdiffusive structure. This setting also shares connections with classical studies of random walks in cones and the quarter-plane, such as \cite{Denwac, Fayolle}, due to the geometric constraints induced by the axes.

Let us first provide more details about the model, denoted by $\mathbf{X} = (X_n)_{n \in \mathbb{N}}$ which starts at the coordinate $(1, 1)$ in $\Z^2$. The transition probabilities of this process differ depending on whether the walk is on the axes, denoted by $K^c := \{(k, m) \in \mathbb{Z}^2 \mid km = 0\}$, or on $K$, which consists of four cones. Specifically, $\mathbf{X}$ is a simple random walk on $K$, meaning that $p(x, x \pm e_i) = \frac{1}{4}$ for all $x \in K$, where $e_i$ is a vector of the canonical basis. However, on $K^c \setminus \{(0, 0)\}$, the walk is directed towards the origin: let $\alpha \geq 0$, for all $i > 0$,
\begin{align*}
    p((i, 0), (i + 1, 0)) &= p((i, 0), (i, \pm 1)) = p((0, i), (0, i + 1)) = p((0, i), (\pm 1, i)) = \frac{1}{4i^{\alpha}}, \\
    p((i, 0), (i - 1, 0)) &= p((0, i), (0, i - 1)) = 1 - \frac{3}{4i^{\alpha}},
\end{align*}
and symmetrically when $i < 0$. Note that the origin is a special point, as $p((0, 0), \pm e_i) = \frac{1}{4}$.
\begin{figure}[htpb]
\centering
\includegraphics[scale=0.6]{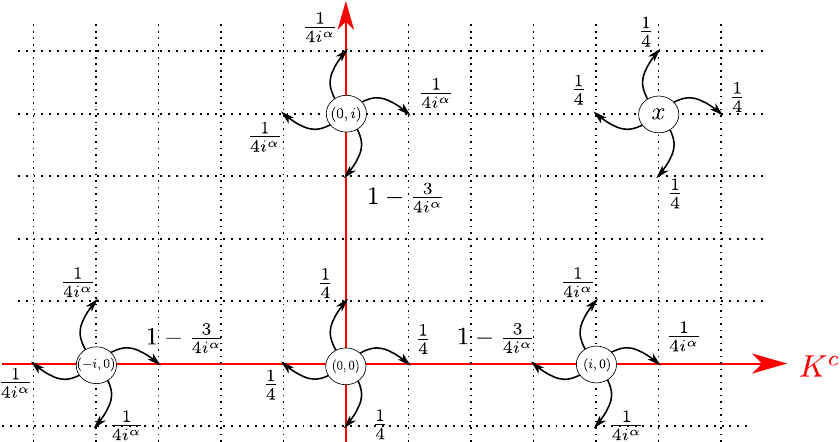}
\caption{probabilities of transition of $\mathbf{X}$ }
\label{whynot} 
\end{figure}
\\

    \noindent The first theorem below is our main result; it provides a description of the behavior of the walk in the quadrants. If $x=(x_1,x_2)$ and $a=(a_1,a_2)$ are in $\mathbb R^2$, we denote $|x|$ for $(|x_1|,|x_2|)$  and $\lbrace x\ge a\rbrace$ for $\{x_1 \geq a_1, x_2 \geq a_2  \}$.

\begin{theo} \label{The1} Assume $\alpha>3$, for any $a=(a_1,a_2) \in (\R_+^*)^2$
\begin{align}
\P\Big( \sqrt{\frac{\log n}{n}} \aXn[n] \geq a\Big)=\int_{0}^{+\infty}\frac{\P(\mB_s \geq 2 a)}{s}\mathrm{d}\mu_{\mathcal{S}_1^0}(s), 
\end{align} 
with $\mB_{s}=(\mB_{s}^1,\mB_s^2)$,  $\mB^1$ and $\mB^2$ are two independent Bessel processes of dimension 3 
and $\mu_{\mathcal{S}_1^0}$ is a 1-stable distribution which is described in the second theorem below.
\end{theo}

In this result, we observe that the normalization combines the diffusive nature of the particle within the cones (evidenced by the appearance of the square root term) and the sequence $n/\log n$ which already appears in the paper \cite{AndDeb3}. The $\log n$ term arises from the tail distribution of the entry time into the axes, starting from a point within the cone, the one-stable distribution also comes from this tail. Finally, the Bessel process of dimension 3 highlights that, most of the time, the process can be viewed as a diffusion conditioned to remain within a cone.

The proof of this theorem actually involves the second order of the law of large numbers for the number of excursions in a time $n$ entering and exiting the axis (the first order has been obtained in \cite{AndDeb3}). 
\noindent To introduce this fact let us introduce the following stopping times,  for any $i \in \mathbb{N}^*$:
\begin{align}
 \eta_i & := \inf \{k>\rho_{i-1},  X_{k} \in K^c \}, \nonumber \\ 
\rho_{i}& := \inf \{k>\eta_{i},  X_{k} \in K \},  \rho_0=0, \nonumber
\end{align}
that is respectively the $i$-th  exit and entrance times in $K$. We also introduce the first exit time from the axis $\rho:=\inf \{k>0,  X_{k} \in K^c \}$ and the first exit time from the cone $\eta:=\inf \{k>0,  X_{k} \in K^c \}$, note that $ \eta_1 = \eta$. Also for any positive integer $n$, let:  
\begin{align}
N_n& := \max\{i \leq n,\ \rho_i \leq n \} \label{last},
\end{align}
which is the number of the last entrance in $K$ before the instant $n$. One of the results in  \cite{AndDeb3} writes : 
 \begin{prop} \textrm{(\cite{AndDeb3})} \label{prop31}
 Assume $\alpha>3$, when $n \rightarrow +\infty$
 \begin{align}
\frac{\log n}{n}N_n  \overset{\P}{ \rightarrow } \frac{\pi}
{8}\frac{1}{ \E(\pi^{\dag}) } =:\co,
\end{align}
where $\pi^{\dag}$ is the invariant probability measure of the positive recurrent Markov chains $(X_{\rho_i})_{i\in \mathbb N}$ and $\E(\pi^{\dag}) :=\sum_{x}\overline x \pi^{\dag}(x)$, equivalently 
\begin{align*}
    \frac{\rho_m}{m \log m} \rightarrow \frac{1}{\co}.
\end{align*}
\end{prop}

\noindent Our second theorem is the joint distribution of the second order for $N_n$ and $n-\rho_{N_n}$. 

\begin{theo} \label{SONn} 
Assume $\alpha>3$, for any $a\geq 0$ and $h \geq 0$, when $n$ goes to infinity
 \begin{align}
\P\left( \frac{(\log n)^2 }{n} \left(N_n- \co \frac{n}{\log n}  \right)\geq h,  \frac{\log n }{n} \left(n-\rho_{N_n}\right) \geq a \right) \rightarrow \int_a^{+ \infty} \frac{1}{s} \mathrm{d}\mu_{\mathcal{S}_1^h}(s) \label{thm2eq2}
\end{align}
where $ \mu_{\mathcal{S}_1^h}$ has the $\log$-Laplace transform, given by for any $\lambda>0$,
\begin{align}
\log \hat{\mu}_{\mathcal{S}_1^h}(\lambda) := \lambda \Big(-\co\E_{\pi^\dag}\Big( \E_{X_{\eta}}( \rho )\Big)+ {\gamma-1}-\frac{h}{\co}    \Big)+   (\lambda-1) \log \lambda ,  \label{LapS_1}
\end{align}
where $\gamma$  is the constant of Euler–Mascheroni. 
\end{theo}
Note that $\frac{1}{s} \mathrm{d}\mu_{\mathcal{S}_1^0}(s)$ is the probability measure of the time elapsed between the last entrance of the random walk on the cone and the instant $n$ after proper normalization, which is $n/\log n$. This result provides a refined description of the fluctuations around the first-order renewal structure obtained in \cite{AndDeb3}. It reveals a subtle second-order structure that intertwines the behavior of the random walk along the axis and inside the cone.

\vspace{0.5cm}
    The rest of the paper is organized as follows, in the section below we prove the above theorem by studying the Laplace transform of the second order of the asymptotic of $\rho_{m}$. In Section 3 we apply the idea of the proof of this result to prove the first theorem. We also add an appendix where we prove technical estimates and collect usual facts on simple random walks.

\section{Proof of Theorem \ref{SONn}}

We assume $\alpha>3$ throughout this section, and those that follow, and do not mention it in the statements of our results.

\subsection{A second order for $(\rho_m)_m$}

In this Section, we study the second order for the random variable $\rho_m$, in particular the convergence in law of $\frac{1}{m}\left( \rho_m -  a_m \right)$  where
\[a_m:= \frac{1}{\co} m \log m.\]
The following corollary is a consequence of Proposition \ref{lem2.3} stated below.  

\begin{cor}\label{propLaplacecondition} \label{corrhom} 
The sequence $(\frac{1}{m}  (\rho_m-a_m))_{m\in\mathbb N^*}$ converges in distribution, the limit distribution is characterized by its log-Laplace given by 
\begin{align} 
 \forall \lambda>0, \ \mathbf{S_1}(\lambda) :=  -\lambda \E_{\pi^\dag}\big( \E_{X_{\eta}}( \rho )\big)+\frac{\lambda}{\co} \big( {\gamma-1} + \log \lambda \big), \label{Stable_1} \end{align}
where $\E_{\pi^\dag}$ is the expectation with respect to probability measure $\pi^\dag$. 
\end{cor}

Such computations are standard in the case of independent random variables; however, this is not the scenario here. Indeed, in this context, $\rho_m$ is generated from a Markov chain, and the classical techniques developed by Gnedenko-Kolmogorov \cite{GneKol} prove inadequate. The approach also differs from the scenario in which random variables are nearly independent, as in \cite{AndDev}, where a renewal structure is also studied. An effective approach in this context is to initially examine the probability limit of this Laplace transform, conditioned on the Markov chain $(X_{\rho_i})_{i\in\mathbb N}$, and subsequently derive the final convergence from it.

The idea of the proof is to understand the probability behavior of a conditional log-Laplace transform of $\rho_m$.
For this purpose, for any $m\geq 1$, we need to introduce the function $A_m$ defined for $x \in K$ and $y \in \partial K$ by: 
\begin{align}
A_m(x,y):= \E_x(\eta \un_{ \eta \leq m} |{X_{ \rho_1}=y}), \mbox{ where } \ \rho_1:=\inf\{k>\eta,  X_k\in K\}.  \label{Am}
\end{align}
 $A_m$ is in fact involved in the first-order asymptotic behavior of $\rho_m$, note indeed the presence of $\eta \un_{\eta \leq m}$ which produces the $\log m$ obtained in \cite{AndDeb3}, recalled here in Proposition \ref{prop31}. We first apply the strong Markov property, 
\begin{align}
&\E\left(e^{- \frac{\lambda}{m} (\rho_m-\sum_{j=1}^{m} A_m(X_{\rho_{j-1}},X_{\rho_{j}}) )}\Big|X_{\rho_i}, 1\le i\le m \right) =e^{\sum_{j=1}^{m} F_m(X_{\rho_{j-1}},X_{\rho_{j}})}\label{CondLap}
\end{align}
where  
\begin{align}  
F_m(x,y):=\log \E_{x}\left(e^{- \frac{\lambda}{m} (\rho_1 -A_m(x,y)) }\Big| X_{\rho_1}=y\right) 
\label{Fm}  
\end{align} 
and recall that $\rho_0=0$ and $X_0=(1,1)$. From this expression, it becomes clear that a key aspect is the approximation of $\sum_{j=1}^{m} F_m(X_{\rho_{j-1}}, X_{\rho_{j}})$, which is the main result of this section.

\begin{prop} \label{lem2.3} Let $d>0$,  the following limit in $m \rightarrow + \infty$ holds in probability 
\begin{align}\label{ergo}
& \sum_{j=1 }^{dm} F_{m}(X_{\rho_{j-1}},X_{\rho_{j}}) \rightarrow d\mathbf{S_1}(\lambda),
\end{align}
note that we assume $dm$ to be an integer.
\end{prop}

The proof of this proposition relies on two lemmata: the main one below studies the Laplace transform of $\rho_1/m$ and its proof, being technical, is postponed in the Appendix. Let us first introduce some notations:

\begin{nota} \label{nota1}
For all $z=(z_1,z_2)\in \Z$, $\oz:=\max(|z_1|,|z_2|)$. In addition, the notation $z>y$ with $z,y\in K^c$ means that the scalar product $<z,y>>0$ and $\oz>\oy$. Note that if $y=(0,0)$, we write that $z>y$ for all $z\in K^c\backslash\{(0,0)\}$.\\
For typographical simplicity, we expend this notation if $y\in\partial K$, more precisely in this case $z>y$ (resp. $z<y$) means that $y=y^\prime\pm e_i$ with $y^\prime\in K^c$ and $z>y^\prime$ (resp. $z<y^\prime$).\\
Moreover we write $\P_{\infty}(X_{\rho}=y)$ for $\lim_{\oz \rightarrow \infty+, z>y}\P_{z}(X_{\rho}=y)$. Note that this limit exists and is strictly positive as $(\P_{z}(X_{\rho}=y))_{z>y}$ is a non-increasing sequence bounded from below by $\frac{1}{4\oy^\alpha}\prod_{z>y}(1-\frac{3}{4\oz^\alpha})>0$, as $\alpha>1$.
\end{nota}

\noindent The following Lemma is proved in Section \ref{techlem},

\begin{lem} \label{leLem} 
Let $x\in K$, $y \in \partial K$, when $m$ goes to infinity, uniformly in $\ox=o(m^{1/2})$ and $\oy\le m^{\nicefrac{1}{2}-\varepsilon}$, for any $\varepsilon>0$ small enough:
\begin{align*} &\E_{x}\left(e^{- \frac{\lambda}{m} \rho_1 }\un_{X_{\rho_1}=y,X_{\eta_1}\le m^{\nicefrac{1}{2}+\varepsilon}}\right)- \P_x(X_{\rho_1}=y)+ \frac{\lambda}{m} \E_x(\eta \un_{\eta \leq m} \un_{X_{\rho_1}=y}) \\
& =-\frac{\lambda}{m}\E_x( \E_{X_{\eta}}(\rho\un_{X_{\rho}=y}) )+ \frac{\lambda}{m} \frac{8 \ox}{\pi } ( \gamma-1+  \log \lambda ) \P_{\infty}(X_{\rho}=y) + o\Big(\frac{\ox}{m}\Big),
\end{align*}
where $\gamma$ is the constant of Euler–Mascheroni.
\end{lem}

\noindent In order to apply the previous lemma to prove Proposition \ref{lem2.3}, we need the following result which states that during the first $m$ excursions, $(X_{\rho_i})_{i \leq m}$ and $(X_{\eta_i})_{i \leq m}$ cannot respectively exceed $m^{1/2-\varepsilon}$ and $m^{1/2+\varepsilon}$ for any small positive $\varepsilon$.

\begin{lem} \label{uniform} 
Let $\varepsilon>0$ and introduce 
\begin{align}
\mathcal{B}_m^{c}:=\left\{\exists i\le m, \oX_{\rho_i} > m^{1/2-\varepsilon} \mbox{ or }\oX_{\eta_i} > m^{1/2+\varepsilon}\right\}.\label{B_n}
\end{align}
Then, for $\varepsilon>0$ small enough:  
\[\lim_{m\rightarrow+\infty}\P( \mathcal{B}_m^{c}) \label{Event1}  = 0. \]   
\end{lem}

\begin{proof}
First, note that: 
\begin{align*}
\P( \mathcal{B}_m^{c})\le \sum_{i=1}^m\P\Big( \oX_{\rho_i} > m^{1/2-\varepsilon}\Big)+\P\Big(\oX_{\eta_i}> m^{1/2+\varepsilon},  \oX_{\rho_{i-1}}\le m^{\nicefrac{1}{2}-\varepsilon}\Big)=:I+II.
\end{align*}
According to \eqref{lem4.6paper1}, there exists $c^\prime>0$ such that for all $i$ and all $x\in\partial K$, 
$\P(X_{\rho_i}=x)\le \nicefrac{c^\prime}{\ox^\alpha}$ implying that for $m$ large enough:
\begin{align*}
I\le \sum_{i=1}^m\frac{8(\alpha-1)^{-1}c^\prime}{m^{(\nicefrac{1}{2}-\varepsilon)(\alpha-1)}}=\frac{8(\alpha-1)^{-1}c^\prime}{m^{(\nicefrac{1}{2}-\varepsilon)(\alpha-1)-1}},
\end{align*}
which tends to 0 when $m$ goes to infinity for $\varepsilon$ small enough as $\alpha>3$.\\ 
Using \eqref{remark5.2paper1} and \eqref{lem4.6paper1}, there exists two positive constants $C$ and $C^\prime$ such that: 
\begin{align*}
II&= \sum_{i=1}^{m}\E\Big[\mathds{1}_{\oX_{\rho_{i-1}}\le m^{\nicefrac{1}{2}-\varepsilon}}\P_{X_{\rho_{i-1}}}\Big(\oX_{\eta}> m^{1/2+\varepsilon}\Big)\Big]\\
&\le C\sum_{i=1}^{m}\E\Big[\mathds{1}_{\oX_{\rho_{i-1}}\le m^{\nicefrac{1}{2}-\varepsilon}}\frac{\oX_{\rho_{i-1}}}{m^{1+2\varepsilon}}\Big]
\le \frac{C}{m^{1+2\varepsilon}}\sum_{i=1}^{m}\E\Big[\oX_{\rho_{i-1}}\Big]\le \frac{C^\prime}{m^{2\varepsilon}},
\end{align*}
which tends to 0 when $m$ goes to infinity.
\end{proof}

\noindent     We will now proceed to prove Proposition  \ref{lem2.3}. Note that we do so for the case  $d =1$, as the reasoning for the general case follows similarly. According to Lemma \ref{uniform}, we can apply Lemma \ref{leLem} to $F_m(X_{\rho_{j-1}},X_{\rho_j})$ and as a result, uniformly for $j\le m$, in probability: 
\begin{align*}
F_m(X_{\rho_{j-1}},X_{\rho_j}) 
&=:\frac{\lambda}{m}\widetilde F(X_{\rho_{j-1}},X_{\rho_j})+o\Big(\frac{\oX_{\rho_{j-1}}}{m}\Big),
\end{align*}
where 
\begin{align*}
\widetilde F(x,y)=\frac{1}{\P_x(X_{\rho_1}=y)}\left(-\E_x(\E_{X_{\eta}}(\rho\un_{X_{ \rho}=y} )+\frac{8\ox}{\pi}(\gamma-1+\log \lambda)\P_\infty(X_\rho=y)\right).
\end{align*}
In lemma 3.2 in \cite{AndDeb3}, we have proved that $(X_{\rho_j})_{j\in\mathbb N}$ is positive recurrent with invariant probability measure $\pi^{\dag}$. Thus, assuming that we can apply the Birkhof ergodic theorem to the function $\widetilde F$, if $\pun$ is the kernel of the Markov chain $(X_{\rho_j})_{j\in\mathbb N}$, we have in probability:
\begin{align}
\lim_{m\rightarrow+\infty}\frac{1}{m}\sum_{i=1}^m\widetilde F&(X_{\rho_{j-1}},X_{\rho_j})
\sum_{x\in \partial K} \sum_{y\in \partial K } \pi^{\dag}(x) \pun (x,y)\widetilde F(x,y)\nonumber\\
&=- \sum_{x\in\partial K} \pi^{\dag}(x) {\E_{x}\big( \E_{X_{\eta}}( \rho ) \big)}+\cst  \E(\pi^{\dag})(\gamma-1+ \log \lambda)  \label{sumergo} .
\end{align}

The existence of two positive constants $C_0$ and $C_1$ such that $\E_{X_{\eta}}( \rho ) \leq C_0 \oX_{\eta} $ (see Lemma 4.8 in \cite{AndDeb3}), and $\E_x(\oX_{\eta}) \leq C_1 \ox$ (see \cite{McConnell} Theorem 1.3 page 223) ensures that the sum in \eqref{sumergo} is bounded above by $\E[\pi^\dag]$ up to a positive constant which is finite according to Lemma 3.2 in \cite{AndDeb3} and justifies the utilization of the Birkhof ergodic theorem. To conclude, we need to account for the term $o\left(\nicefrac{\oX_{\rho_{j-1}}}{m}\right)$ in the expression of $F_m$, but again, the Birkhoff ergodic theorem ensures that $\frac{1}{m} \sum_{i=1}^m \oX_{\rho_{j-1}}$ converges almost surely to a constant.  
So we obtain \eqref{ergo} recalling that $(\co)^{-1}=\nicefrac{8  \E(\pi^{\dag})}{\pi}$.
\newline

\noindent We are now ready to prove Corollary \ref{corrhom}.

\noindent We first check that $\sum_{j=1}^{m} A_m(X_{\rho_{j-1}},X_{\rho_{j}})-a_m$ is small,  this actually comes from the first order, that is the law of large numbers of $\rho_m$ obtained in \cite{AndDeb3}. Indeed  in the proof of Lemma 3.3 in \cite{AndDeb3}, it is proved that for any $d>0$  
\begin{align}
\E(e^{-\lambda \rho_{dm}/m})=\E(e^{-\lambda d \log m (\rho_{dm} /(d m \log m)) })=  e^{-\frac{\lambda d}{\cun} \log m +O(1) }, \label{papier1}
\end{align}
using a reasoning similar to that used in the proof of Proposition \ref{lem2.3}:
\begin{align}
 \E(e^{-\lambda \rho_{dm} /m})
&= \E\Big(e^{-\frac{\lambda d}{d m} \sum_{j=1 }^{d m} A_{m} (X_{\rho_{j-1}},X_{\rho_{j}}) +O(1) }\Big).  \label{papier2}
\end{align}
Identifying \eqref{papier1} and \eqref{papier2} implies that 
\begin{align}
  \sum_{j=1 }^{dm} A_{m} (X_{\rho_{j-1}},X_{\rho_{j}})  
 &=_{\P}  \frac{d}{\cun}  m\log m+o(m), \label{Aconv}
\end{align}
where $=_{\P}$
 means that the equality holds with  probability tending to one when $m$ goes to infinity, taking $d=1$, we obtain what we seek.  
Writing: 
\begin{align*}
    \rho_m-a_m &= \Big(\rho_m-\sum_{j=1}^m A_{m}(X_{\rho_{j-1}},X_{\rho_{j}})\Big) + \Big(\sum_{j=1}^m A_{m}(X_{\rho_{j-1}},X_{\rho_{j}})-a_m\Big),
\end{align*} 
and using \eqref{CondLap}, we have: 
 \begin{align*}
\E(e^{-\frac{\lambda}{m}(\rho_m-a_m)})=\E\Big(e^{\sum_{j=1}^mF_m(X_{\rho_{j-1}},X_{\rho_j})-\frac{\lambda}{m}(\sum_{j=1}^mA_m(X_{\rho_{j-1}},X_{\rho_j})-a_m)}\Big).
 \end{align*}
Formula \eqref{Aconv} and Proposition \ref{lem2.3} ensure that the quantity in the exponential above converges to $\mathbf{S_1}(\lambda)$ when $m$ goes to infinity in probability, and as a result in distribution. Consequently, we have our corollary as: 
\[\lim_{m\rightarrow+\infty}\E(e^{-\frac{\lambda}{m}(\rho_m-a_m)})=e^{\mathbf{S_1}(\lambda)}.\]

\noindent We finish this section with the following remark which is a particular case of the preceding corollary and will be used in the proof of Theorem \ref{The1}. Let us introduce the positive sequence $(b_n)_{n\ge 2}$ defined by:
\begin{align}\label{defbninutile}
b_n^2=\frac{n}{\log n}+\frac{n \log \log n}{(\log n)^2}.
\end{align}
It follows that when $n$ goes to infinity: 
\begin{align}  
b_n^2 \log b_n^2 =n+o\left(\frac{n}{\log n}\right), \label{defbn} 
\end{align}
and note that necessarily $b_n^2=n/\log n +o(n/\log n)$.

\begin{rem} \label{rem2.3} In the sequel, in order to simplify notations, we proceed as if $\co b_n^2$ were an integer. For any $\varepsilon>0$, introduce the event $\mathcal{G}_n$ defined by
\begin{align}
\Big\{ \Big| \frac{\lambda}{b_n^2} \Big(n- \sum_{j=1 }^{\co b_n^2} A_{b_n^2}(X_{\rho_{j-1}},X_{\rho_{j}})\Big)-\sum_{j=1 }^{\co b_n^2} F_{b_n^2}(X_{\rho_{j-1}},X_{\rho_{j}}) +\cun \mathbf{S_1}(\lambda)\Big|\leq \varepsilon \Big\}, 
\label{Gn}
\end{align}
we have $\lim_{n \rightarrow + \infty}\P(\mathcal{G}_n)=1$. Indeed, replacing {$m$ by $b_n^2$} and $d$ by $\cun$ in \eqref{Aconv}, we have:
\begin{align*}
 \sum_{j=1 }^{\co b_n^2} A_{ b_n^2} (X_{\rho_{j-1}},X_{\rho_{j}})&
 =_{\P} \frac{\co}{\co}  b_n^2   \log (b_n^2)+o(b_n^2)  =n+o(b_n^2)
\end{align*}
where the last inequality comes from the definition of $b_n^2$. Then the result follows by applying Proposition \ref{lem2.3}.
\end{rem}

\subsection{Proof of Theorem \ref{SONn}}

Let $a>0$, $h>0$ and introduce $h_n:= h b_n^2 (\log n)^{-1}$. 
First, we decompose the probability in \eqref{thm2eq2} with respect to the values of $\eta_{N_n+1}$ and introduce the event $\mathcal{G}_n$ defined in \eqref{Gn} which depends only on $\{X_{\rho_1},\cdots,X_{\rho_{\co b_n^2}}\}$:
\begin{align*}
&\P( n-\rho_{N_n} \geq a b_n^2, N_n- \co b_n^2\geq h_n)\\
&=\P( n-\rho_{N_n} \geq a b_n^2,N_n- \co b_n^2\geq h_n, \eta_{N_n+1}>n,\mathcal{G}_n)\\
&+\P( n-\rho_{N_n} \geq a b_n^2,N_n- \co b_n^2\geq h_n, \eta_{N_n+1}\leq n,\mathcal{G}_n)+o(1)=:P_1+P_2+o(1). 
\end{align*} 
The event in $P_1$ means that, after the instant $\rho_{N_n}$, the random walk will not leave the cone before the instant $n$. In contrast, the event in $P_2$ indicates that it will exit the cone at some point and remain on the axis until the instant $n$.\\
We first deal with the main contribution $P_1$. 
Note that according to lemma \ref{uniform}, as $N_n\le n$, we have:
\begin{align*}
P_1&=\P( n-\rho_{N_n} \geq a b_n^2,N_n- \co b_n^2\geq h_n, \eta_{N_n+1}>n,X_{\rho_{N_n}}\le n^{\nicefrac{1}{2}-\varepsilon},\mathcal{G}_n)+o(1)\\
&=:P_1^\prime+o(1). 
\end{align*}
with: 
\begin{align*}
P_1^\prime
&=\sum_{i-\co b_n^2 \geq h_n }\sum_{x,\ox\le n^{\nicefrac{1}{2}-\varepsilon}} \sum_{m \geq  a b_n^2  } \P(X_{\rho_i}=x, n-\rho_i=m,{\mathcal{G}_n})\P_x(\eta >m)\\
&=\frac{8}{\pi}   \sum_{m \geq  a b_n^2  } \sum_{\ox \leq n^{1/2- \varepsilon}} \frac{\ox}{m}  \sum_{i- \co b_n^2 \geq h_n} \P(X_{\rho_i}=x, n-\rho_i=m,{\mathcal{G}_n}) +o(1)
\end{align*}
where the last equality comes from  \eqref{lleta2} as indeed $m\ge ab_n^2$ implies $\ox=o(m^{\nicefrac{1}{2}})$.\\
Note that, throughout the calculations, $i$ is bounded above by $n$, but to avoid complicating the notation, we will not write this out in the rest of the text.\\
One can see that, using \eqref{upperbound1}, there exists a positive constant $C$, such that:
\begin{align*}
& \sum_{i- \co b_n^2 \geq h_n}  \sum_{\ox > n^{1/2- \varepsilon}} \ox\P(X_{\rho_i}=x)\sum_{m\ge ab_n^2}\frac{1}{m}\P(n-\rho_i=m,|X_{\rho_i}=x, )\\
 &\le\frac{1}{ab_n^2} \sum_{i- \co b_n^2 \geq h_n} \sum_{\ox > n^{1/2- \varepsilon}}\frac{c_+}{\ox^{\alpha-1}}\le \frac{C\log n}{n^{(\nicefrac{1}{2}-\varepsilon)(\alpha-1)-1}}
\end{align*}
which tends to 0 for $\varepsilon$ small enough when $n$ goes to infinity as $\alpha>3$. Consequently, we can write:
\begin{align}\label{dmun} 
P_1^\prime=\frac{8}{\pi}\sum_{m \geq  a b_n^2 }  \frac{1}{m}  \E\Big[ \sum_{i-\co b_n^2 \geq h_n} \oX_{\rho_i} \E\big(\un_{ n-\rho_i=m}\big|(X_{\rho_j},j \leq i)\big) \un_{{\mathcal{G}_n}} \Big] +o(1).
\end{align}

\noindent Then, the main idea to evaluate $P_1'$, is to prove that for large $n$, under $\mathcal{G}_n$, the above second sum can be approximated in probability by a certain (random) measure. Then we determine the mean of this measure. 

\noindent For this purpose, we introduce the random measure $\mu^h_{n,X}$ and its mean $\mu_n^h$, defined for any $s>0$ by:
\begin{align}
    \mu^h_{n,X}([0,s]) &:= {\frac{8}{\pi}} \frac{(\log n)^2}{n}\sum_{i-\co b_n^2 \geq h_n} \oX_{\rho_i} \E(\un_{ 0 \leq n-\rho_i \leq s b_n^2 }|(X_{\rho_j},j \leq i))\un_{{\mathcal{G}_n}}, \nonumber \\
    \mu_n^h([0,s])& :=\E(\mu^h_{n,X}(s)) \label{Mesure},
\end{align}
To simplify the notations we write $ \mu^h_{n,X}(s)$ (resp. $\mu_n^h(s)$) instead of $\mu^h_{n,X} $$([0,s])$ (resp. $\mu_n^h([0,s])$).

So, we first obtain (in probability) an estimate of the conditioned Laplace transform of the above conditional measure. For this, we will use, in particular, Remark \ref{rem2.3}. This will lead to the vague convergence of $\mu_n^h$ to the measure $\mu_{\mathcal{S}_1^h}$, for which we have an explicit log-Laplace transform (see  \eqref{LapS_1}). 
We summarize this convergence in the following proposition, which we assume to be true for the moment:

\begin{prop} \label{Vague} For any $h\geq 0$, $ \mu_n^h$ converges vaguely when $n$ goes to infinity to $\mu_{\mathcal{S}_1^h}$, where $\mu_{\mathcal{S}_1^h}$ is characterized by its $\log$-Laplace transform (see \eqref{LapS_1}).
\end{prop}

\noindent  We deduce from this proposition the following convergence 
\begin{align*}
\lim_{n \rightarrow + \infty } P_1=\lim_{n \rightarrow + \infty } P_1'  = \int_a^{+ \infty} \frac{1}{s} \mathrm{d}\mu_{\mathcal{S}_1^h}(s). 
\end{align*}

\noindent Let us now prove the above Proposition and then we will show that the remaining probability $P_2$ is negligible.

\begin{proof} of Proposition \ref{Vague}. Let $\hat \mu^h_{n,(X_{\rho_k})_k}(\lambda)$, the Laplace transform of $\mu^h_{n,(X_{\rho_k})_k}$, following the same lines as in \eqref{CondLap}, it can be written as
\begin{align}
    & \hat \mu^h_{n,X}(\lambda)= \nonumber \\
    &  \cst \frac{(\log n)^2}{n} \sum_{i-\co b_n^2 \geq h_n } \oX_{\rho_i}  e^{\frac{\lambda }{b_n^2} (n-  \sum_{j\leq i} A_{b_n^2}(X_{\rho_{j-1}},X_{\rho_{j}})) -\sum_{j \leq i}F_{b_n^2}(X_{\rho_{j-1}},X_{\rho_{j}})  } \un_{\mathcal{G}_n}. \nonumber
\end{align}
with recall \eqref{Am} and \eqref{Fm}. As $i- \co b_n^2 \geq h_n$, we introduce 
the following random variables 
\begin{align*}
    \Delta F_{b_n^2}(i)  :=  \sum_{j=\co b_n^2+1}^{i}  F_{b_n^2}(X_{\rho_{j-1}},X_{\rho_{j}}),\, \Delta A_{b_n^2}(i)  :=\sum_{j=\co b_n^2+1}^{i}  A_{b_n^2}(X_{\rho_{j-1}},X_{\rho_{j}}).
\end{align*}
Then, focusing on the exponent of the exponential $\hat \mu^h_{n,X}(\lambda)$, as we work under $\mathcal{G}_n$ (see Remark \ref{rem2.3}), we can write: 
\begin{align*}
& \frac{\lambda}{b_n^2} ( n- \sum_{j\leq i} A_{b_n^2}(X_{\rho_{j-1}},X_{\rho_{j}})) -\sum_{j \leq i}  F_{b_n^2}(X_{\rho_{j-1}},X_{\rho_{j}})  \\
& =\frac{\lambda}{b_n^2} ( n-  \sum_{j=1 }^{\co b_n^2} A_{b_n^2}(X_{\rho_{j-1}},X_{\rho_{j}}))-\sum_{j=1 }^{\co b_n^2}  F_{b_n^2}(X_{\rho_{j-1}},X_{\rho_{j}})\\
&-\frac{\lambda}{b_n^2}\Delta A_{b_n^2}(i)-\Delta F_{b_n^2}(i)\\
&=\cun \mathbf{S_1}(\lambda)-\frac{\lambda}{b_n^2}\Delta A_{b_n^2}(i)-\Delta F_{b_n^2}(i)+o(1).
\end{align*}
So we are left to deal with 
\begin{align}
 {\frac{8}{\pi}} \frac{(\log n)^2}{n} \sum_{i-\co b_n^2 \geq h_n } \oX_{\rho_i}  e^{  -\frac{\lambda}{b_n^2}\Delta A_{b_n^2}(i)-\Delta F_{b_n^2}(i)}. \label{letruc}
\end{align}
We first focus on $\Delta F_{b_n^2}(i)$ and with a similar reasoning as the one for the proof of Proposition \ref{lem2.3}, one can see that for $n$ large enough:  
\begin{align}
\Delta F_{b_n^2}(i)=_\P\frac{i-\cun b_n^2}{b_n^2}S_1(\lambda)+o(1)\label{Delta1}.
\end{align} 
Now, we prove that  $\Delta F_{b_n^2}(i)$ produces actually a negligible contribution compared to $\frac{1}{b_n^2}\Delta A_{b_n^2}(i)$: using \eqref{Aconv}, we obtain for $n$ large enough, 
\begin{align}\label{Delta2}
\Delta A_{b_n^2}(i)=\sum_{j=\co b_n^2+1}^{i} A_{b_n^2}(X_{\rho_{j-1}},X_{\rho_{j}})=_{\P} (i-\co b_n^2)\frac{\log n}{\co}+o(i-\co b_n^2).
\end{align}

\noindent Recalling that $b_n^2  \sim \nicefrac{n}{\log n}$, when $n$ goes to infinity, \eqref{Delta1} and \eqref{Delta2} imply: 
\begin{align*}
\frac{\lambda}{b_n^2}\Delta A_{b_n^2}(i)+\Delta F_{b_n^2}(i)
&=_{\P}\frac{i-\co b_n^2}{b_n^2}\Big(\mathbf{S_1}(\lambda)+\frac{\lambda}{\co}\log n+o(1)\Big)+o(1)\\ 
&=\frac{\lambda(\log n)^2}{n\co}(i-\co b_n^2)(1+o(1)).
\end{align*}
As a result \eqref{letruc} can be approximated (in probability) to:
\begin{align}\label{lenouveautruc}
 D_n:=  {\frac{8}{\pi}} \frac{(\log n)^2}{n} \sum_{i-\co b_n^2 \geq h_n } \oX_{\rho_i}  e^{  - \frac{\lambda}{\co} \frac{(\log n)^2}{n} (i-\co b_n^2)} .
\end{align} 
Our aim at this step is to obtain the limit of $\E(D_n)$. As $(X_{\rho_i})_{i\ge 1}$ is positive recurrent and aperiodic, $\E(\oX_{\rho_i})$ converges to $\E(\pi^{\dag})$ and thus:
\begin{align*}
\lim_{n\rightarrow+\infty}\E(D_n)&=\frac{8}{\pi}\E(\pi^\dagger)\lim_{n\rightarrow+\infty}\frac{(\log n)^2}{n}\sum_{i-\co b_n^2 \geq h_n } e^{  -\frac{\lambda }{\co} \frac{(\log n)^2}{n} (i-\co b_n^2)}\\
&=\co\lim_{n\rightarrow+\infty}\frac{(\log n)^2}{n}\frac{e^{  -\frac{\lambda }{\co} \frac{(\log n)^2}{n} h_n}}{1-e^{  -\frac{\lambda }{\co} \frac{(\log n)^2}{n} }} =\frac{e^{-\frac{\lambda h}{\co}}}{\lambda}.
\end{align*}
Moving back to \eqref{letruc} we finish the proof by an argument of dominated convergence which yields : 
\begin{align*}
\lim_{n \rightarrow + \infty} \E\Big( {\frac{8}{\pi}} \frac{(\log n)^2}{n} \sum_{i-\co b_n^2 \geq h_n } \oX_{\rho_i}  e^{  -\frac{\lambda}{b_n^2}\Delta A_{b_n^2}(i)-\Delta F_{b_n^2}(i)}\Big) =  \frac{e^{-\frac{\lambda h}{\co}}}{{\lambda}}.
\end{align*}
\noindent Then collecting the above limit together with the arguments above \eqref{letruc}, we obtain 
\begin{align*}
\lim_{n\rightarrow+\infty}\hat \mu_{n}(\lambda)  
=\frac{e^{\cun \mathbf{S_1}(\lambda)} e^{-\frac{\lambda h}{\co}}}{\lambda}, 
\end{align*}
this finishes the proof of the Proposition \ref{Vague}.
\end{proof}

It remains to prove that $P_2$ (see page  8)  provides a negligible part. The main idea regarding how this leads to a small contribution is as follows: to end up on the axis without exiting it before time $n$, one must simultaneously undergo a significant fluctuation within the cone while ensuring that the remaining time on the axis (after this major excursion) is not greater than the norm of this fluctuation.
This actually occurs with low probability.\\
{First, note that according to lemma \ref{uniform}, as $\eta_{N_n+1}\le n$, we have:
\begin{align*}
P_2=\P( n-\rho_{N_n} \geq a b_n^2,N_n- \co b_n^2\geq h_n, \eta_{N_n+1}\leq n,\mathcal B_n,\mathcal{G}_n)+o(1)=:P_2^\prime+o(1).
\end{align*}
Similarly as for $P_1$:
\begin{align*}
P_2^\prime
=& \sum_{i-\co b_n^2 \geq h_n }\sum_{x,\ox \leq n^{1/2-\varepsilon}} \sum_{m \geq  {a b_n^2} }\P(X_{\rho_i}=x, n-\rho_i=m,\mathcal G_n) \\
& \sum_{y, \oy \leq n^{1/2+\varepsilon} } \sum_{u=1}^{m-1} \P_x(\eta =u,X_{\eta}=y)\P_y(\rho>m-u).
\end{align*}
Then we decompose $P_2^\prime$ into two parts $\Sigma_1$ and $\Sigma_2$, respectively when $m-u> b_n^{1+3\varepsilon}$ or not.\\
In the first case, as $\oy \leq n^{1/2+ \varepsilon}$, this implies that $\oy=o(b_n^{1+3\varepsilon})$ for $\varepsilon$ small enough, then by \eqref{lem4.7paper1}
, for any $r>0$,  $\P_y(\rho>m-u) \leq b_n^{-(1+3\varepsilon)r}$ and therefore (as $i\le n$):
\begin{align*}
\Sigma_1
\le&\frac{1}{b_n^{(1+3\varepsilon)r}}\sum_{i-\co b_n^2 \geq h_n }\P(\oX_{\rho_i}\le n^{\nicefrac{1}{2}-\varepsilon}, n-\rho_i\ge ab_n^2,\mathcal G_n)
\le\frac{n}{b_n^{(1+3\varepsilon)r}},
\end{align*}
implying that $\Sigma_1$ goes to 0 when $n$ goes to infinity for $r$ large enough.\\
For $\Sigma_2$, as $u>ab_n^2-b_n^{1+3\varepsilon}$, $\ox=o(\sqrt u)$ for $\varepsilon$ small enough and according to \eqref{lleta}, $\P_x(\eta=u)\le \frac{16\ox}{\pi u^2}$. Consequently, for $n$ large enough: 
\begin{align*}
\Sigma_2\le&\sum_{i-\co b_n^2 \geq h_n }\sum_{x,\ox \leq n^{1/2-\varepsilon}} \sum_{m \geq  {a b_n^2} }\P(X_{\rho_i}=x, n-\rho_i=m) \sum_{u= m-b_n^{1+3\varepsilon} }^{m-1} \P_x(\eta =u) \\
\le&\frac{16}{\pi}\sum_{i-\co b_n^2 \geq h_n }\sum_{x,\ox \leq n^{1/2-\varepsilon}} \sum_{m \geq  {a b_n^2} }\P(X_{\rho_i}=x, n-\rho_i=m) \sum_{u= m-b_n^{1+3\varepsilon} }^{m-1} \frac{\ox}{u^2}\\
\le&\frac{16b_n^{1+3\varepsilon}}{\pi(ab_n^2-b_n^{1+3\varepsilon})^2}\sum_{i-\co b_n^2 \geq h_n }\sum_{x,\ox \leq n^{1/2-\varepsilon}} \sum_{m \geq  {a b_n^2} }\ox\P(X_{\rho_i}=x, n-\rho_i=m)\\
\le&\frac{32}{\pi a^2b_n^{3(1-\varepsilon)}}\sum_{i-\co b_n^2 \geq h_n }\E\Big[\oX_{\rho_i}\mathds{1}_{\frac{n-\rho_i}{b_n^2}\ge a,\mathcal G_n}\Big].
\end{align*}}

In the last expression, we recognize the measure introduced in \eqref{Mesure}, which implies that $\sum_{i-\co b_n^2 \geq h_n } $ $ \E(X_{\rho_i}\mathds{1}_{\rho_i \leq n-a b_n^2})$ behaves like $n/(\log n)^2$, implying that $\Sigma_2$ converges to zero when $n$ goes to $+\infty$.
Collecting these two estimates ($\sum_1$ and $\sum_2$), we obtain that $P_2$ converges to zero when $n$ goes to $+\infty$.

\section{Convergence of the main process}

In this section, we prove Theorem \ref{The1}. We start with a first section of preliminaries, then we study the limit of $\P(X_n \geq a b_n,X_n \in K)$, with $a=(a_1,a_2)$ and $a_1>0$ and $a_2>0$. 

\subsection{Preliminaries}

In this part, we prove with the following lemma that for $n$ large enough, the random walk $\mathbf X$ cannot be at the same time on the axis and far from the origin. 
\begin{lem} \label{lemneg}
If $e_i$ is a vector of the canonical basis, the probability of the event $\mathcal{A}_n:=\Big \{ \aXn[n] \geq b_ne_i ,X_n \in K^c \Big\}$ 
tends to 0 when $n$ goes to infinity.
\end{lem}
\begin{proof} 
Let $\varepsilon>0$, $I_n:= b_n^2(\varepsilon , \varepsilon^{-1} )$, and  $H_n:=\frac{b_n^2}{\log n}( \varepsilon , \varepsilon^{-1})$, Theorem \ref{SONn} implies that ${\lim_{\varepsilon \rightarrow 0}} \lim_{n \rightarrow +\infty} \P(N_n -\co b_n^2 \in H_n, n-\rho_{N_n} \in I_n)=1$, using also Lemma \ref{uniform} 
\begin{align*} 
\P(\mathcal{A}_n) 
 =& \sum_{i-\co b_n^2 \in H_n }\sum_{x,\ox \leq n^{1/2-\varepsilon}} \sum_{m \in I_n } \P(X_{\rho_i}=x,n- \rho_i=m) \times \\
& \P_x(|X_{m}| \geq  b_n e_i,\eta_1 \leq m, \rho_1>m)+o_{\varepsilon}(1), 
\end{align*}
with $\lim_{\varepsilon \rightarrow 0} \lim_{n \rightarrow +\infty}o_{\varepsilon}(1)=0$.\\
Let us examine the second probability, introducing $r_n:=b_n-b_n/ \log n$ and we decompose it as follows:
\begin{align*}
&\P_x(|X_{m}| \geq b_n e_i,\eta_1 \leq m, \rho_1>m)\\
&=\sum_{u=1}^{m} \sum_{y\in K^c} \P_x(\eta=u,X_{u}=y)\P_y(|X_{m-u}|\geq  b_ne_i, \rho>m-u) \\
&=\sum_{u=1}^{m}\Big( \sum_{y, \oy \geq r_n}+ \sum_{y, \oy < r_n}\Big) \P_x(\eta=u,X_{u}=y)\P_y(|X_{m-u}|\geq b_ne_i, \rho>m-u)\\
&=:I+II.
\end{align*}
For $II$, as $\oy < r_n$ and we want $X_{m-u}\geq b_n e_i$, there exists $0\le \ell\le m-u-2$ such that $|X_\ell|=(b_n-2)e_i,\,|X_{\ell+1}|=(b_n-1)e_i$ and $|X_{\ell+2}|=b_ne_i$. Consequently, there exists a positive constant $C_1>0$ such that: 
\begin{align}\label{borneII}
II\le 
\sum_{u=1}^{m} \sum_{y, \oy < r_n} \P_x(\eta=u,X_{u}=y)\frac{C_1n}{b_n^{2\alpha}} \leq \frac{C_1n}{b_n^{2\alpha}}.
\end{align}
$I$ is more delicate and we decompose with respect to the values of $u$, introducing $c_n:=n/ (\log n)^{2-\varepsilon}=o(b_n^2)$ for small $\varepsilon>0$. We write: 
\begin{align*}
I&=\Big( \sum_{u<c_n}+\sum_{u=c_n}^{m}\Big)\sum_{y, \oy \geq r_n}\P_x(\eta=u,X_{u}=y)\P_y(|X_{m-u}|\geq b_ne_i, \rho>m-u)\\
&=:III+IV.
\end{align*}
As $n^{1/2-\varepsilon}=o(\sqrt{c_n})$, $\ox \leq n^{1/2-\varepsilon}$ and $m \geq \varepsilon n/ (\log n)^2$, we can use  \eqref{lleta2}, implying the existence of a positive constant $C_2$ such that for $n$ large enough:
\begin{align}\label{borneIV} 
IV&
\leq \sum_{u=c_n}^{m} \sum_{\oy \geq r_n} \P_x(\eta=u,X_{u}=y)
\leq  \P_x(\eta>c_n)\le C_2\ox\frac{(\log n)^{2-\varepsilon}}{n}.
\end{align}
For $III$, as $\ox=o(r_n)$ and $u<c_n$, then for any $u<c_n$, according to \eqref{locallimit}, there exists a positive constant $C_3$ such that for $n$ large enough:
\[\P_x(\eta=u,X_{u}\geq r_n) \leq C_3 \ox \frac{e^{-\frac{r_n^2}{2 u}}}{u^2},\]
implying the existence of $C_4>0$, such that for $n$ large enough:
\begin{align}\label{borneIII}
III&\leq \sum_{u \leq  c_n} \P_x(\eta=u,X_{u}\geq r_n)\le C_4\ox \frac{e^{-\frac{(\log n)^{1-\varepsilon}}{2}}}{r_n^2}=o\left(\ox\frac{(\log n)^{2-\varepsilon}}{n}\right).
\end{align}

By collecting \eqref{borneII},\eqref{borneIII} and \eqref{borneIV}, there exists a positive constant $C$ such that for $n$ large enough:
\begin{align*}
\P(\mathcal{A}_n) 
& \leq \sum_{i-\co b_n^2 \in H_n }\sum_x \sum_{m \in I_n } \P(X_{\rho_i}=x, n-\rho_i=m) 
 C\Big(\frac{n}{b_n^{2\alpha}}+\ox\frac{(\log n)^{2-\varepsilon}}{n}\Big) +o(1) \\
&\leq 2C\Big( \frac{n}{\varepsilon \log n }\frac{1}{b_n^{2(\alpha-1)}}+ \frac{(\log n)^{2-\varepsilon}}{n}\sum_{i-\co b_n^2 \in H_n } \E(X_{\rho_i})\Big).
\end{align*}
As $(X_{\rho_i})_i$ is positive recurrent with invariant probability measure $\pi^{\dag}$ and $\E(\pi^{\dag})<+\infty$, then for $n$ large enough $\sum_{i-\co b_n^2 \in H_n } \E(X_{\rho_i}) \leq 2 \E(\pi^{\dag}) \frac{1}{\varepsilon}\frac{n}{(\log n)^2}$.\\
Finally, as $\alpha>3$,  for $n$ large enough, there exists $C_5>0$ such that $\P(\mathcal{A}_n)\leq  C_5 (\log n)^{- \varepsilon}$ which converges to 0.
\end{proof}

\subsection{The limit of $\P(\aXn[n]\geq a b_n)$, with $a=(a_1,a_2)\in (\R_+^*)^2$}

 Since excursions play a central role, we decompose, as in the proof of Lemma \ref{lemneg}, the probability $\mathbb{P}(\aXn[n] \geq a b_n)$ with respect to the number of excursions $N_n$ and the time elapsed $n-\rho_{N_n}$. The reasoning is similar to the one used for the computation of $P_1$ in the proof of Lemma \ref{lemneg}, and as a result, we skip some details to avoid making the text more cumbersome (using in particular Theorem \ref{SONn} and Lemma \ref{uniform}). Let us introduce  $\tilde b_n:=b_n^2/\log n$, and recall the events $\mathcal{G}_n$ in \eqref{Gn}, then for any $\varepsilon>0$ and $\varepsilon'>0$:
\begin{align*}
\P(\aXn[n] \geq a b_n)&= \P(\aXn[n] \geq a b_n,\mathcal{G}_n,{\oX_{\rho_{N_n}}\le n^{\nicefrac{1}{2}-\varepsilon}})+o(1)\\
&=  \sum_{x, \ox \leq n^{1/2- \varepsilon}} \sum_{m \in I_n } \P_x(\aXn[m] \geq a b_n, \eta >m)\\ 
& \hfill \sum_{i-\co b_n^2 \geq \varepsilon' \tilde b_n } \P(X_{\rho_i}=x, n-\rho_i=m,\mathcal{G}_n) +o_{\varepsilon,\varepsilon'}(1) \\ 
& =  {\frac{16}{\pi} (1+o(1))  \sum_{m \in I_n }\frac{1}{m} \P(\mB_{{m/b_n}} \geq 2a )} \\
& \hfill \sum_{\ox \leq n^{1/2- \varepsilon}} \ox\sum_{i-\co b_n^2 \geq \varepsilon' \tilde b_n } \P(X_{\rho_i}=x, n-\rho_i=m,\mathcal{G}_n) +o_{\varepsilon,\varepsilon'}(1)
\end{align*}  
where the last inequality comes from \eqref{Bessel3} and $o_{\varepsilon,\varepsilon'}(1)$ is such that $\lim_{\varepsilon' \rightarrow 0}$ $\lim_{\varepsilon \rightarrow 0}\lim_{n \rightarrow +\infty}o_{\varepsilon,\varepsilon'}(1)=0$. 
And with the same idea used to obtain \eqref{dmun}, we have:
\begin{align}
&   \frac{16}{\pi}\sum_{m \in I_n } \frac{1}{m} \P(\mB_{{m/b_n}} \geq 2a ) \times  \nonumber \\
& \sum_{\ox \leq n^{1/2- \varepsilon}} \ox \sum_{i-\co b_n^2 \geq \varepsilon' \tilde b_n} \P(X_{\rho_i}=x, n-\rho_i=m,\mathcal{G}_n) \nonumber \\
& =  \frac{16}{\pi} \sum_{m \in I_n }  \frac{1}{m} \P(\mB_{{m/b_n}} \geq 2a ) \E\Big(\sum_{i-\co b_n^2 \geq \varepsilon' \tilde b_n } \oX_{\rho_i}\un_{ n-\rho_i=m,\mathcal{G}_n}\Big)+o_{\varepsilon,\varepsilon'}(1), \label{thesum}
\end{align}

Then we see appear the measure $\mu_n^{\varepsilon'}$ (introduced  in \eqref{Mesure}), its vague convergence (Proposition \ref{Vague}) together with the distribution of $\mB_{{m/b_n}}$ (see \eqref{Bessel3}) implies that the above expression converges when $n$ goes to infinity to  

\begin{align*}
 \int_{\varepsilon }^{1/\varepsilon} \frac{1}{s} \P(\mB_{ s } \geq 2 a) \mathrm{d}\mu_{\mathcal{S}_1^{\varepsilon'}}(s).
\end{align*}
Taking the limit when $\varepsilon'$ and $\varepsilon$ goes to zero, we obtain Theorem \ref{The1}.

\section{Appendix}
The appendix is essentially devoted to the proof of the Laplace transform of $\rho_1$ with constraints, that is, the proof of the following Lemma: 

\subsection{Proof of Lemma \ref{leLem} \label{techlem}}
Let us first recall this lemma : \\
{\bf Lemma \ref{leLem}} \textit{ Let $x\in K$, $y \in \partial K$, when $m$ goes to infinity, uniformly in $\ox=o(m^{1/2})$ and $\oy\le m^{\nicefrac{1}{2}-\varepsilon}$, for any $\varepsilon>0$ small enough:
\begin{align*} &\E_{x}\left(e^{- \frac{\lambda}{m} \rho_1 }\un_{X_{\rho_1}=y,X_{\eta_1}\le m^{\nicefrac{1}{2}+\varepsilon}}\right)- \P_x(X_{\rho_1}=y)+ \frac{\lambda}{m} \E_x(\eta \un_{\eta \leq m} \un_{X_{\rho_1}=y}) \\
& =-\frac{\lambda}{m}\E_x( \E_{X_{\eta}}(\rho\un_{X_{\rho}=y}) )+ \frac{\lambda}{m} \frac{8 \ox}{\pi } ( \gamma-1+  \log \lambda ) \P_{\infty}(X_{\rho}=y) + o\Big(\frac{\ox}{m}\Big),
\end{align*}}

\begin{proof}
By the strong Markov property, one can write: 
\begin{align*}
 \E_{x}\left(e^{- \frac{\lambda}{m} \rho_1 } \un_{X_{\rho_1}=y,\oX_{\eta_1}\le m^{\nicefrac{1}{2}+\varepsilon}}\right ) = \sum_{z\in K^c,\oz\le m^{\nicefrac{1}{2}+\varepsilon}}\E_{z}\left(e^{- \frac{\lambda}{m} \rho } \un_{X_{\rho}=y} \right ) \E_{x}\left(e^{- \frac{\lambda}{m} \eta } \un_{X_{\eta}=z}\right).
\end{align*}
   
As $\E[\rho^2]<\infty$ (see \eqref{lem4.7paper1}) and $e^{-x}<1-x+\frac{x^2}{2}$ for all $x>0$, when $m$ goes to infinity:

\[ \E_{z}\left(e^{- \frac{\lambda}{m} \rho } \un_{X_{\rho}=y} \right )=\P_z(X_{\rho}=y)-\frac{\lambda}{m}\E_z(\rho\mathds{1}_{X_{\rho}=y})+o\Big(\frac{1}{m}\Big) .\]

Recall also that in \cite{AndDeb3} (Lemma 5.4) the study of the tail of $\eta$ is obtained; however, here we need more details as we want the second order. Integrating by parts the Laplace transform of $\eta$ gives:
\begin{align}
   \E_{x}\left(e^{- \frac{\lambda}{m} \eta } \un_{X_{\eta}=z}\right)=\P_x(X_{\eta}=z)  - \lambda \int_0^{ +\infty} e^{-\lambda u}\P_x(\eta>um, X_{\eta}=z)\mathrm{d}u. \label{lintegr}
\end{align}
Note that,
\begin{align*}
 & \sum_{z\in K^c,\oz\le m^{\nicefrac{1}{2}+\varepsilon}}\P_x(X_{\eta}=z)\Big(  \P_z(X_{\rho}=y)-\frac{\lambda}{m}\E_z(\rho\mathds{1}_{X_{\rho}=y})+o\Big(\frac{1}{m}\Big) \Big)  \\ 
 & = \P_x(X_{\rho_1}=y)-\frac{\lambda}{m} \E_x(\E_{X_\eta}(\rho \un_{X_{\rho}=y}))+o\Big(\frac{1}{m}\Big),
\end{align*}
indeed by applying \eqref{momrho} and \eqref{locallimitX},  $\sum_{z>m^{\nicefrac{1}{2}+\varepsilon}}\P_x(X_{\eta}=z)\E_z(\rho\mathds{1}_{X_{\rho}=y})=o(1)$.
Thinking similarly, we also process the integral in \eqref{lintegr} and finally obtain:
\begin{align}
 &\E_{x}\left(e^{- \frac{\lambda}{m} \rho_1 } \un_{X_{\rho_1}=y}\right )-\P_x(X_{\rho_1}=y)+\frac{\lambda}{m} \E_x(\E_{X_\eta}(\rho \un_{X_{\rho}=y})) \nonumber \\
 &= -\lambda \sum_{z\in K^c, \oz\le m^{\nicefrac{1}{2}+\varepsilon}}\P_z(X_{\rho}=y) \int_0^{+\infty } e^{-\lambda u}\P_x(\eta>um, X_{\eta}=z)\mathrm{d}u \nonumber\\
& +\frac{\lambda^2}{m}\sum_{z\in K^c, \oz\le m^{\nicefrac{1}{2}+\varepsilon}}\E_z(\rho\mathds{1}_{X_{\rho}=y}) \int_0^{ +\infty} e^{-\lambda u}\P_x(\eta>um, X_{\eta}=z)\mathrm{d}u+o\Big(\frac{1}{m}\Big) \nonumber\\
 &= -\lambda \sum_{z\in K^c}\P_z(X_{\rho}=y) \int_0^{+\infty } e^{-\lambda u}\P_x(\eta>um, X_{\eta}=z)\mathrm{d}u \nonumber\\
& +\frac{\lambda^2}{m}\sum_{z\in K^c, \oz\le m^{\nicefrac{1}{2}+\varepsilon}}\E_z(\rho\mathds{1}_{X_{\rho}=y}) \int_0^{ +\infty} e^{-\lambda u}\P_x(\eta>um, X_{\eta}=z)\mathrm{d}u+o\Big(\frac{1}{m}\Big) \nonumber \\
&=:-\lambda A+\frac{\lambda^2}{m}B+o\Big(\frac{1}{m}\Big). \label{AB}
\end{align}
The primary contribution stems from $A$, so we begin by studying this quantity.\\
\noindent $\bullet$ \textit{Contribution of A,} we first decompose the integral into two, let $C>0$ an integer:
\begin{align*}
A& = \sum_{z\in K^c} \P_z(X_{\rho}=y) I_1(z) + \sum_{z\in K^c} \P_z(X_{\rho}=y) I_2(z) =: \sum_{ \leq C}+\sum_{ > C}
\end{align*}
with $I_1(z)=: \int_0^C e^{-\lambda u}\P_x(\eta>um, X_{\eta}=z)\mathrm{d}u$ and $I_2(z)=:\int_C^{+\infty} e^{-\lambda u}\P_x(\eta>um, X_{\eta}=z)\mathrm{d}u$. 
\\
Let us prove that the contribution of $\sum_{ > C}$, is negligible. \\ 
\noindent {\bf Step 1 : contribution of $\sum_{ > C}$} \\
First, we write: 
\begin{align*}
     \sum_{z\in K^c}\P_z(X_{\rho}=y) I_2(z)  & \le   \sum_{z\in K^c} \int_C^{+\infty } e^{-\lambda u}\P_x(\eta>um, X_{\eta} =z )\mathrm{d}u\\
  & \le \int_C^{+\infty } e^{-\lambda u}\P_x(\eta>um)\mathrm{d}u \le \P_x(\eta>Cm )\frac{e^{-\lambda C}}{\lambda}.
\end{align*}
Then as $\ox =o(m^{1/2})$, according to \eqref{lleta2}:
\begin{equation}
\P_x(\eta>Cm)=\frac{8}{\pi} \frac{ \ox }{Cm}(1+o(1)), 
\end{equation}
implying that $\sum_{>C}$ is bounded above by $\frac{16 \ox }{m} \frac{e^{-\lambda C}}{\lambda C }$. For large $C$ this part will therefore lead to a negligible contribution compared to $\ox/m$.

\medskip
\noindent {\bf Step 2 : contribution of $\sum_{ \leq  C}$} \\
We first rewrite $I_1(z)$ using the fact that $\eta$ is an integer:
\begin{align*}
    & \int_0^C e^{-\lambda u}\P_x(\eta>um, X_{\eta}=z)\mathrm{d}u=\sum_{i=0}^{{Cm-1}} \int_{i/m}^{(i+1)/m} e^{-\lambda u}\P_x(\eta>u m,  X_{\eta}=z)\mathrm{d}u \\ 
    & = \sum_{i=0}^{Cm-1} \P_x(\eta> i , X_{\eta}=z) \int_{i/m}^{(i+1)/m} e^{-\lambda u}\mathrm{d}u \\
    & = \frac{1}{\lambda} \sum_{i=0}^{Cm-1} \P_x(\eta> i , X_{\eta}=z) e^{-\lambda i /m }(1- e^{-\lambda /m } ) \\
    & = \frac{1}{m} (1-O(1/m)) \sum_{i=0}^{Cm-1} \P_x(\eta> i , X_{\eta}=z) e^{-\lambda i /m }
\end{align*}
So to treat $\sum_{ \leq C}$, we are interested in the following sum,
\begin{align*}
\sum_{z\in K^c }\P_z(X_{\rho}=y)\sum_{i=0}^{Cm-1} \P_x(\eta> i , X_{\eta}=z) e^{-\lambda i /m }. 
\end{align*}
In the previous sums, we decompose the sum over $i$ into two parts: 
\[\sum_{z\in K^c }\sum_{i=0}^{Cm-1} \cdots =\sum_{z\in K^c }\Big(\sum_{i \leq b_m^2} \cdots + \sum_{i= b_m^2+1}^{Cm-1} \cdots\Big) =: \Sigma_1+\Sigma_2,\]
where $b_m^2$ is defined in \eqref{defbninutile} and is equivalent to $\nicefrac{m}{\log m}$ as $m$ goes to infinity.\\
 \textit{$\bullet$ For $\Sigma_2$}, we decompose this sum into two terms depending on the values of $z$:
 \[\Sigma_2:= \sum_{\oz \geq m^{1/2-\varepsilon}}+ \sum_{\oz < m^{1/2-\varepsilon}}=: \Sigma_2^1+\Sigma_2^2.\] 
\textit{We first deal with $\Sigma_2^1$}, 
as  $\oz \geq \oy$ is large, $\P_z(X_{\rho}=y)= \P_{\infty}(X_{\rho}=y)(1+o(1))$ (see Notations \ref{nota1}), and with a reasoning similar to the one that shows $\sum_{>C}$ to be negligible:
\begin{align*}
\Sigma_2^1
&= \P_{\infty}(X_{\rho}=y)(1+o(1)) \sum_{i=b_m^2+1}^{Cm-1} \P_x(\eta>i, \oX_{\eta} \geq  m^{1/2-\varepsilon} ) e^{-\lambda i /m }\\
&= \P_{\infty}(X_{\rho}=y)(1+o(1)) \sum_{i=b_m^2+1}^{\infty} \P_x(\eta>i, \oX_{\eta} \geq  m^{1/2-\varepsilon} ) e^{-\lambda i /m }+o(1).
\end{align*}
As $ m^{1/2-\varepsilon}=o(b_m)$, we have $\P_x(\eta>i, \oX_{\eta} \geq  m^{1/2-\varepsilon} )  =\frac{8\ox}{i\pi}+o\left(\frac{\ox}{i}\right)$ (see \eqref{estimate}) and thus:
\begin{align*}
& \sum_{i=b_m^2+1}^{\infty} \P_x(\eta>i, \oX_{\eta} \geq  m^{1/2-\varepsilon} ) e^{-\lambda i /m } \\
&=\frac{ 8 }{\pi} \ox \Big(-\log(1-e^{-\lambda/m})-\sum_{i=1}^{b_m^2}\frac{e^{-\lambda i/m}}{i}\Big)+o(\ox).
\end{align*}
The first term yields $\log(1-e^{-\lambda/m})= -\log m+\log \lambda+o(1)$. The second term $\sum_{i=1}^{b_m^2}\frac{e^{-\lambda i/m}}{i}=\sum_{i=1}^{b_m^2}\frac{1}{i}+o(1)=\log b_m^2+\gamma+o(1) $ where $\gamma$ is the constant of Euler–Mascheroni. 
So we obtain:
\begin{align*}
\Sigma_2^1=\P_{\infty}(X_{\rho}=y)(1+o(1))\frac{8}{\pi} \ox( \log m  - \log \lambda -\log b_m^2- \gamma+o(1)).
\end{align*}
For further purpose, we would like to make the truncated mean of $\eta$ appear instead of $\log m$ and $\log b_m^2$, so recall (see \eqref{lleta3})  
that for large $m$, and $\ox =o(m^{1/2})$, $\E_x(\eta \un_{\eta \leq m})=\frac{8}{\pi} \ox \log m+o(1) $. So finally 
\begin{align}
\Sigma_2^1=\P_{\infty}(X_{\rho}=y) \Big( \E_x(\eta \un_{\eta \leq m})-\E_x(\eta \un_{\eta \leq b_m^2}) -\frac{8}{\pi} \ox(\log \lambda+ \gamma)\Big)+o(\ox). \label{Eeta}
\end{align}
 \textit{For $\Sigma_2^2$}, as $\oz\le m^{1/2-\varepsilon}$, using  \eqref{estimate2}, for $m$ large enough:
\begin{align*}
    \Sigma_2^2\le \sum_{i\ge b_m^2}\P_x(\eta>i, \oX_\eta\le m^{1/2-\varepsilon})\le \frac{4\ox }{\pi}m^{1-2\varepsilon}\sum_{i\ge b_m^2}\frac{1}{i^2}\le \frac{8\ox }{\pi}\frac{\log m}{m^{2\varepsilon}}=o(1).
\end{align*} 
\noindent  \textit{$\bullet$ For $\Sigma_1$}, as $i \leq b_m^2$ we can get rid of the $e^{- \lambda i/m}$, also similarly as for $\Sigma_2$, we decompose $\Sigma_1$ with respect to the values of $z$:
  \begin{align}
    \Sigma_1 \sim \sum_{z\in K^c }\P_z(X_{\rho}=y)\sum_{i\leq b_m^2} \P_x(\eta> i , X_{\eta}=z) =: \Sigma_1^1+\Sigma_1^2.
\end{align}
\noindent 
For $\Sigma_1^2$,  we would like to merge it with the first expectation in \eqref{Eeta}. First we just rewrite $\Sigma_1^2$ : 
\begin{align*}
\Sigma_1^2 & =  \sum_{\oz\leq m^{1/2-\varepsilon}} \P_{z}(X_{\rho}=y)  \sum_{i=0}^{b_m^2}\P_x(\eta>i, X_\eta=z) \\
&=  \sum_{\oz\leq m^{1/2-\varepsilon}} \P_{z}(X_{\rho}=y)  \sum_{i=0}^{b_m^2}(\P_x(b_m^2\ge \eta>i, X_\eta=z)+\P_x(\eta>b_m^2, X_\eta=z) )\\
&=  \sum_{\oz\leq m^{1/2-\varepsilon}} \P_{z}(X_{\rho}=y)  (\E_x(\eta\mathds{1}_{\eta\le b_m^2, X_\eta=z})+(b_m^2+1)\P_x(\eta>b_m^2, X_\eta=z) )\\
&=:  \sum_{\oz\leq m^{1/2-\varepsilon}} \P_{z}(X_{\rho}=y)  \E_x(\eta\mathds{1}_{\eta\le b_m^2, X_\eta=z})+\mathscr B_m.
\end{align*}
Then, using \eqref{estimate2}, there exists $C>0$ such that for $m$ large enough:
\begin{align*}
\mathscr B_m\le (b_m^2+1)\P_x(\eta>b_m^2, \oX_\eta\le m^{\nicefrac{1}{2}-\varepsilon})\le C\frac{ \ox m^{1-2\varepsilon}}{b_m^2}\sim C \frac{\ox \log m}{m^{2\varepsilon}}=o(\ox).
\end{align*}
Thus
\begin{align*}
\Sigma_1^2&=\sum_{\oz\leq m^{1/2-\varepsilon}} \P_{z}(X_{\rho}=y)  \E_x(\eta\mathds{1}_{\eta\le b_m^2, X_\eta=z})+o(\ox)\\
&-\sum_{\oz\leq m^{1/2-\varepsilon}} \P_{z}(X_{\rho}=y)  \E_x(\eta \un_{b_m^2 \leq \eta \leq m, X_{\eta}=z})+o(\ox) \\
&= \sum_{\oz\leq m^{1/2-\varepsilon}} \P_{z}(X_{\rho}=y)  \E_x(\eta \un_{\eta \leq m, X_{\eta}=z}) +o(\ox).
\end{align*}
where the last equality comes from the fact that the second sum above is small due to the contradictory condition that $z$ is small but $\eta$ is large (see \eqref{locallimit}).
Using again that for $\oz >m^{1/2-\varepsilon}$, $\P_z(X_{\rho}=y)= \P_{\infty}(X_{\rho}=y)(1+o(1))$, we write:
 \begin{align*}
\P_{\infty}(X_{\rho}=y) \E_x(\eta \un_{\eta \leq m})& =\sum_{\oz> m^{1/2-\varepsilon}} \P_{z}(X_{\rho}=y) \E_x(\eta \un_{\eta \leq m, X_{\eta}=z}) \\
&+ \P_{\infty}(X_{\rho}=y)\E_x(\eta \un_{\eta \leq m, \oX_{\eta}\leq m^{1/2-\varepsilon}}).
\end{align*}
Then we can see that the main contribution of $\Sigma_1^2 $ and the first sum above are complementary,  
\begin{align*}
&\Sigma_1^2 +\P_{\infty}(X_{\rho}=y) \E_x(\eta \un_{\eta \leq m})  \\ &
=\E_x(\eta \un_{\eta \leq m,X_{\rho_1}=y})+ \P_{\infty}(X_{\rho}=y)  \E_x(\eta \un_{\eta \leq m, \oX_{\eta}\le m^{1/2-\varepsilon}})+o(\ox).
\end{align*}
\noindent Finally for $\Sigma_1^1$, as $\oz>\oy$ and large, with a similar reasoning as the one for $\Sigma_1^2$, using \eqref{estimate}:
\begin{align*}
&    \Sigma_1^1=(1+o(1))\P_{\infty}(X_{\rho}=y)\E_x(\eta\un_{\eta\leq b_m^2,\oX_{\eta}> m^{1/2-\varepsilon}})\\
&+(1+o(1))\P_{\infty}(X_{\rho}=y)(b_m^2+1)\P_x(\eta>b_m^2, \oX_\eta>m^{\nicefrac{1}{2}-\varepsilon})\\
&  =(1+o(1))\P_{\infty}(X_{\rho}=y)\Big(\E_x(\eta\un_{\eta\leq b_m^2,\oX_{\eta}>m^{1/2-\varepsilon}})+\frac{8}{\pi}\ox \Big)+o(\ox).
\end{align*} 
To summarize, by rearranging the terms in $\Sigma_1 + \Sigma_2 = \Sigma_1^1 + \Sigma_1^2 + \Sigma_2$, it can be written as:
\begin{align*} 
\Sigma_1+\Sigma_2& =\E_x(\eta \un_{\eta \leq m,X_{\rho_1}=y}) -\P_{\infty}(X_{\rho}=y) \frac{8}{\pi} \ox(\log \lambda+ \gamma-1)  \\
&+\P_{\infty}(X_{\rho}=y)(\E_x(\eta\un_{\eta \leq b_m^2,\oX_{\eta}>  m^{1/2-\varepsilon}})-   \E_x(\eta \un_{\eta \leq b_m^2}))\\
& +\P_{\infty}(X_{\rho}=y)  \E_x(\eta \un_{\eta \leq m, \oX_{\eta}\le m^{1/2-\varepsilon}})+o(\ox)\\
&=\E_x(\eta \un_{\eta \leq m,X_{\rho_1}=y}) -\P_{\infty}(X_{\rho}=y) \frac{8}{\pi} \ox(\log \lambda+ \gamma-1)  \\
&+\P_{\infty}(X_{\rho}=y)  \E_x(\eta \un_{b_m^2<\eta \leq m, \oX_{\eta}\le m^{1/2-\varepsilon}})+o(x)\\
&=\E_x(\eta \un_{\eta \leq m,X_{\rho_1}=y}) -\P_{\infty}(X_{\rho}=y) \frac{8}{\pi} \ox(\log \lambda+ \gamma-1)+o(\ox).
\end{align*}
And note that the last expectation is a $o(1)$ as the two constraints $X_{\eta}$ small and $\eta$ large appear in the expression.\\
We are left to prove that the term with $B$ in \eqref{AB}, is negligible.\\
$\bullet$ {\it Contribution of $B$}:
\[\frac{\lambda^2}{m}B=\frac{\lambda^2}{m}\sum_{z\in K^c, \oz\le m^{\nicefrac{1}{2}+\varepsilon}}\E_z(\rho\mathds{1}_{X_{\rho}=y}) \int_0^{ +\infty} e^{-\lambda u}\P_x(\eta>um, X_{\eta}=z)\mathrm{d}u.\]
Note that it suffices to prove that there exists a positive constant $C$ and $0<\beta<1$ such that for all $\oz\le m^{1/2+\varepsilon}$,
\begin{align} \E_z[\rho\mathds{1}_{X_{\rho}=y}]\le C \P_z(X_\rho=y)m^\beta. \label{eq33b}\end{align}
Writing: 
\begin{align*}
\E_z[\rho\mathds{1}_{X_{\rho}=y}]&=\sum_{k>0}\P_z(\rho>k, X_\rho=y)\\
&=\sum_{k\leq m^\beta}\P_z(\rho>k, X_\rho=y)+\sum_{k>m^\beta}\P_z(\rho>k, X_\rho=y)\\
&=:C_1+C_2.
\end{align*}
First, note that $C_1\le m^\beta \P_z(X_\rho=y)$ and as a result, we just have to treat $C_2$.\\
For typographical simplicity, we introduce $y^\prime\in K^c$ such that $y=y^\prime\pm e_i$, where $e_i$ is a vector of the canonical basis (if $\oy=1$, we do not have the uniqueness of $y^\prime$). As a result, $\oy=\overline{y^\prime}$ and one can write: 
\begin{equation}\label{direct}
\P_z(X_\rho=y)\ge\P(z\rightarrow y^\prime)\frac{1}{4\oy^\alpha}, 
\end{equation}
where $\P(z\rightarrow y^\prime)$ is the probability that starting from $z$ the walk follows the shortest path along the axes from  $z$ to $y^\prime$.\\
Now, we split into two cases: when $z$ is in $L_{y}:=\lbrace z\in K^c, \oz\ge \oy,<z,y^\prime>>0\rbrace$ or not.\\
\underline{Case $z\in L_y$:}
As $\oy\le m^{\nicefrac{1}{2}-\varepsilon}$, using \eqref{direct}:
\begin{align*}
\P_z(X_\rho=y) \geq  \frac{1}{4\oy^{\alpha}}\prod_{w>\oy}\left(1-\frac{3}{4\ow^\alpha}\right)\ge \frac{C}{\oy^\alpha} \geq \frac{C}{ m^{\alpha(1/2-\varepsilon)}}.
\end{align*}
Now taking $\beta=\nicefrac{1}{2}+2\varepsilon$, we have $\oz=o(m^\beta)$ and using \eqref{lem4.7paper1}, for any $r>0$, there exists a positive constant $c$:
\begin{align*}
 \sum_{k>m^\beta}\P_z(\rho>k, X_\rho=y)\leq  \sum_{k>m^\beta}\P_z(\rho>k) \leq \sum_{k>m^\beta} \frac{1}{k^r} \leq \frac{c}{m^{\beta(r-1)}}.
\end{align*}
Taking $r=\alpha+1$, one can see that:
\begin{align*}
 \sum_{k>m^\beta}\P_z(\rho>k, X_\rho=y)\leq \P_z(X_\rho=y).
\end{align*}
\underline{Case $z\notin L_y$:}
Take a path $\Gamma$ from $z$ to $y^\prime$ on $K^c$ of length $q-1$, $\Gamma:=(x_0=z,\,x_1,\,\dots,\, x_{q-2}, x_{q-1}=y^\prime)$. Its probability is
\[\P_z(X_1=x_1,\dots, X_{q-1}=y^\prime)=\prod_{i=0}^{q-2}p(x_i,x_{i+1})=\P(z\rightarrow y^\prime) A_{\Gamma}\]
and note that $A_{\Gamma}$ is a product such that if $p(x_i, x_{i+1})$ appears in $A_{\Gamma}$ there is also necessarily $j\ne i$, such that $p(x_j, x_{j+1})=p(x_{i+1}, x_{i})$.\\ 
One can note that the probability of each of this loop satisfies that: 
\begin{equation}
p(x_i, x_{i+1})p(x_i, x_{i+1})\le \frac{1}{4}\max\left(\frac{3}{4},1-\frac{3}{2^{\alpha+2}}\right)=:\frac{\omega_\alpha}{4}.
\end{equation}
A classical result states that the number of paths from $z$ to $y^\prime$ of length $q-1$ is $\binom{q-1}{\frac{q-\ell_{zy^\prime}-1}{2}}$ where $\ell_{zy^\prime}$ denotes the length of the shortest path from $z$ to $y^\prime$. More precisely, as $z\notin L_y$: 
\[\ell_{zy^\prime}=(\oy-\oz)\mathds{1}_{<y^\prime,z>>0}+(\oy+\oz)\mathds{1}_{<y^\prime,z>\le 0}.\]
Note that $\frac{q-\ell_{zy^\prime}-1}{2}$ is also the number of loops, as a result, using the fact that $\oz=o(m^\beta)$ and $\oy=o(m^\beta)$, for $m$ large enough: 
\begin{align*}
\P_z(X_\rho=y, \rho=q)&\le \P(z\rightarrow y^\prime)\frac{1}{4\oy^\alpha}\binom{q-1}{\frac{q-\ell_{zy^\prime}-1}{2}}\left(\frac{\omega_\alpha}{4}\right)^{\frac{q-\ell_{zy^\prime}-1}{2}}\\
&\le \P(z\rightarrow y^\prime)\frac{1}{4\oy^\alpha}\omega_\alpha^{\frac{q-\ell_{zy^\prime}-1}{2}}\left(\frac{1}{4}\right)^{\frac{-\ell_{zy^\prime}-1}{2}}.
\end{align*}
Using this previous inequality, for $m$ large enough there exists a positive constant $\tilde C$ such that: 
\begin{align*}
 \sum_{k>m^\beta}\P_z(\rho>k, X_\rho=y)&\leq  \tilde C\P(z\rightarrow y^\prime)\frac{1}{4\oy^\alpha}\omega_\alpha^{\frac{m^\beta}{2}}\left(\frac{1}{4}\right)^{\frac{-\ell_{zy^\prime}-1}{2}}\\
 &\leq  \tilde C\P(z\rightarrow y^\prime)\frac{1}{4\oy^\alpha}\\
 &\le \tilde C\P_z(X_\rho=y), 
\end{align*} 
since $\omega_\alpha^{\frac{m^\beta}{2}}4^{\frac{\ell_{zy^\prime}+1}{2}}$ tends to 0 when $m$ goes to infinity as $\ell_{zy^\prime}=o(m^\beta)$.\\
As a result, there exists a positive constant $C$ such that:
\[C_2\le C \P_z(X_\rho=y),\]
and finally
\[\E_z(\rho\mathds{1}_{X_{\rho}=y}) \leq C \P_z(X_\rho=y) m^{1/2+2\varepsilon}.\]
Taking $\varepsilon$ small enough we obtain \ref{eq33b}.
\end{proof}

\subsection{Facts about the random walk in the cone and on the axis}
This section presents a compilation of facts, primarily concerning exit times and their coordinates from a cone or axis. 
\medskip

\noindent {\bf Facts for the random walk on the axis}

\noindent The following facts can be found in Section 4 of \cite{AndDeb3}.

\begin{enumerate}
\item There exists a positive constant $c_+$ such that for all $y\in K^c$ and all $x\in \partial K$:
\begin{equation}\label{upperbound1}
\P_y(X_\rho=x)\le \frac{c_+}{\ox^\alpha}.
\end{equation}
\item Assume $\alpha>2$, there exists $c$ and $c'$  such that for any $i$ and $x \in K$  
\begin{equation}\label{lem4.6paper1}
\E(\oX_{\rho_i}) \leq c \textrm{ and } \P(X_{\rho_i}=x) \leq c'/\ox^{\alpha}.
\end{equation} 
\red{\item For any $x$ such that $\ox$ is large
\begin{equation} \label{momrho} 
\E_{x}(\rho) \sim \ox,
\end{equation}
moreover for any $r>0$ and any $x$ such   $\ox=o(m)$, 
\begin{equation}\label{lem4.7paper1}  \P_{x}(\rho > m) \leq m^{-r}.
\end{equation} }
\end{enumerate}

\noindent {\bf Facts for the random walk in the cones} \\
The following Facts can be found or deduced (see below) in the Section 5 of \cite{AndDeb3}.

\begin{enumerate}
\item  For any $x\in \partial K$ such that $\ox=o(\sqrt{k})$, 
\begin{align}
\lim_{k \rightarrow +\infty} \frac{k^2}{\ox}\P_{x}(\eta=k)= \frac{8}{\pi}. \label{lleta}
\end{align}
Implying that:
\begin{align}
\lim_{k \rightarrow +\infty} \frac{k}{\ox}\P_{x}(\eta>k)=\frac{8}{\pi}, \label{lleta2}
\end{align}
and
\begin{align}
\lim_{k \rightarrow +\infty} \frac{1}{\ox\log k}\E_{x}[\eta\mathds{1}_{\eta <k}]=\frac{8}{\pi}. \label{lleta3}
\end{align}

\item Assume that $x$ and $y$ are in the same quarter plane and $\ox =o(\oy)$, then 
\begin{equation}
  \P_x( X_{\eta}=y) \sim  \frac{16}{\pi} \frac{\ox}{\oy^3}  \label{locallimitX},
\end{equation}
in particular
there exists a positive constant $C$ such that:
\begin{equation}\label{remark5.2paper1}
\P_x(\oX_\eta>\oy)\le C \frac{\ox}{\oy^2},
\end{equation}

also for any large $k$ and $ y \in K^c $,
\begin{equation}
  \P_x(\eta=k, X_{\eta}=y) \sim \frac{2}{\pi} \frac{\ox \cdot \oy}{k^3} e^{-(\ox^2+\oy^2)/2k }  \label{locallimit},
\end{equation}

moreover : 
\begin{equation}\label{estimate}
  \textrm{if } \oy=o(\sqrt{m}),\  \P_x(\eta> m, \oX_{\eta}\ge \oy) \sim \frac{8\ox}{\pi m}.
\end{equation}
{Otherwise if $x=(x_1,x_2)$ with $x_1>0$, $x_2>0$ such that  $\ox=o(\sqrt{m})$, and $a_1>0$ and $a_2>0$ then for any $s>0$ 
\begin{align}
   & \P_x(\eta> sm, X_{\eta}\ge (a_1 \sqrt m,a_2 \sqrt m) ) \nonumber \\ 
   & \sim \frac{16\ox}{\pi m} e^{-a_1^2/s}e^{-a_2^2/s} =\frac{16\ox}{\pi m} \P( \mB_{s} \geq (2a_1,2a_2) ) \label{Bessel3}
\end{align}
with $\mB_{s}=(\mB_{s}^1,\mB_s^2)$ and $\mB^1$ and $\mB^2$ are two independent Bessel processes of dimension 3. }\\
Finally 
\begin{equation}\label{estimate2}
  \P_x(\eta> m,\oX_{\eta}< \oy) \sim \frac{2\ox\,\oy^2 }{\pi m^2}.
\end{equation}
\end{enumerate}

\noindent A few words on the last estimates. In \cite{AndDeb3} we do not have exactly \eqref{locallimit}, but this can be deduced from what is denoted $\sum_{12}$ page 17 of \cite{AndDeb3}. Also we can get \eqref{estimate} from \eqref{locallimit}, indeed 
\begin{align*}
& \P_x(\eta>m, \oX_{\eta} \geq \oy) \sim \frac{2\ox}{\pi }\sum_{k > m} \sum_{z,\oz \geq  \oy} \frac{ \oz}{k^3}e^{-(\ox^2+\oz^2)/2k } \\
& \sim \frac{2\ox}{\pi}\sum_{k > m} \sum_{z,\oz \geq  \oy} \frac{ \oz}{k^3}e^{-\oz^2/2k }= \frac{8\ox}{\pi}\sum_{k > m} \frac{e^{-\oy^2/2k }}{k^2} \\
&\sim \frac{8\ox}{\pi}\sum_{k > m} \frac{1}{k^2}\sim \frac{8}{\pi} \frac{ \ox }{m}
\end{align*}
and also  \eqref{estimate2}, 
\begin{align*}
& \P_x(\eta>m, \oX_{\eta} < \oy) \sim \frac{2\ox}{\pi }\sum_{k > m}  \frac{1}{k^3}\sum_{z,\oz <  \oy} \oz \sim \frac{8\ox}{\pi}\frac{1}{2m^2}\frac{\oy^2}{2}.
\end{align*}

\bibliographystyle{plain}
\bibliography{thbiblio}

\end{document}